\documentclass[leqno]{amsart}
\usepackage{amsmath}
\usepackage{amssymb}
\usepackage{amsthm}
\usepackage{enumerate}
\usepackage[mathscr]{eucal}
\usepackage{graphicx}
\theoremstyle{plain}
\newtheorem{theorem}{Theorem}[section]
\newtheorem{prop}[theorem]{Proposition}
\newtheorem{cor}{Corollary}[theorem]
\newtheorem{lemma}{Lemma}[section]

\theoremstyle{definition}
\newtheorem{definition}{Definition}[section]
\newtheorem{remark}{Remark}[section]

\newtheorem{example}{Example}[theorem]

\usepackage[pagewise]{lineno}

\begin{document}
\title[Norm Inequalities]{Norm Inequalities for Inner Product Type Integral Transformers}
\author[ Benard Okelo ]{ Benard Okelo }

\newcommand{\acr}{\newline\indent}

\address{Benard Okelo\\ School of Mathematics and Actuarial Science, Jaramogi Oginga Odinga University of Science and Technology, Box 210-40601, Bondo, Kenya.}
\email{bnyaare@yahoo.com}

\thanks{ This work was partially supported by the DFG Grant No. 1603991000.}


\begin{abstract}
 In this paper, we give a detailed survey on norm inequalities for inner product type integral transformers. We first consider unitarily invariant norms and  operator valued functions. We then give results on norm inequalities for inner product type integral transformers in terms of Landau  inequality,   Gr\" uss  inequality. Lastly, we explore some of the applications in quantum theory.
\end{abstract}
\maketitle
\section{Introduction}
Let $\mathcal{H}$ be an infinite dimensional  complex Hilbert space and $\mathcal{B(H)}$ the algebra of all bounded linear operators on $\mathcal{H}.$
 In this paper, we discuss various types of norm inequalities for  inner product type integral transformers in terms of Landau type inequality,   Gr\" uss type inequality and Cauchy-Schwarz type inequality. We shall also consider  the applications in quantum theory. We begin by the following definition.
 \begin{definition}
Gr\"uss inequality, states that if $f$ and $g$ are integrable real functions
on $[a,b]$ such that $C\leq f(x)\le D$ and $E\leq g(x)\le F$ hold
for some  real constants $C,D,E,F$ and for all $x\in[a,b]$, then
\begin{equation}
  \left|\frac1{b-a}\int_a^bf(x)g(x)dx-\frac1{(b-a)^2}\int_a^b f(x)dx\int_a^b g(x)dx\right|
\leq\frac14(D-C)(F-E).\label{grisovaca}
\end{equation}
\end{definition}

 Inequality \ref{grisovaca} is very interesting to many researchers and it has beeen considered in many  studies whereby
conditions on functions are varied to give different estimates
(see \cite{JOC} and references therein). More on this inequality (and the classical one in \cite{joc09i}) are discussed in the sequel.

Next, we discuss a very important definition of inner product type integral (i.p.t.i) transformer which is key to our study.
\begin{definition}\label{def2}
Consider weakly$\mu^*$-measurable operator valued (o.v) functions $A, B:\Omega\rightarrow \mathcal{B(H)}$ and
for all $X\in \mathcal{B(H)}$ let the function $t\rightarrow A_t X B_t$ be  also
weakly$\mu^*$-measurable. If these functions are Gel'fand integrable
for all $X\in \mathcal{B(H)}$, then the inner product type  linear
transformation $X\to\int_\Omega A_t X B_t dt$ is be called an
inner product type integral (i.p.t.i) transformer on $\mathcal{B(H)}$ and denoted by
$\int_\Omega A_t \otimes B_t dt$ or ${\mathcal I}_{A,B}$.
\end{definition}
\begin{remark}
When  $\mu$ is the counting measure on $\mathbb N$ then such transformers are
known as  elementary operators whose some of the properties have
been studied in details(see \cite{Oke1} and the references therein on orthogonality property).
 \end{remark}
 This work is organized as follows: Section 1: Introduction; Section 2: Unitarily invariant norms; Section 3: Operator valued functions; Section 4: Norm inequalities for inner product type integral transformers and lastly; Section 5: Applications in quantum theory.

\section{Unitarily invariant norms}
In this section, we consider a special type of norms called the unitarily
invariant norm. We give its description in details which will be useful in the sequel.
Let
$\mathcal{C}_\infty(\mathcal{H})$ denote the space of all compact linear operators acting on a separable,
complex Hilbert space $\mathcal{H}$. Each symmetric gauge function $\Phi,$  simply denoted by (s.g.)
on sequences gives rise to  a unitarily invariant
(u.i) norm on operators defined by
$\left\|X\right\|_\Phi=\Phi(\{s_n(X)\}_{n=1}^\infty)$, with
$s_1(X)\ge s_2(X)\ge\hdots$ being the singular values of $X$, i.e.,
the eigenvalues of $|X|=(X^*X)^\frac12.$ We will denote by the
symbol $\left|\left|\left|\cdot\right|\right|\right|$ any such norm,
which is therefore defined on a naturally associated norm ideal
$\mathcal{C}_{\left|\left|\left|\cdot\right|\right|\right|}(\mathcal{H})$ of
 of $\mathcal{C}_\infty(\mathcal{H})$ and  satisfies the
invariance property $ |\|UXV|\|=|\|X|\|$ for all
$X\in\mathcal{C}_{\left|\left|\left|\cdot\right|\right|\right|}(\mathcal{H})$
and for all unitary operators $U,V\in \mathcal{B(H)}$.
One of the  well known among u.i.  norms are the Schatten $p$-norms
defined for $1\le p<\infty$ as $\|X\|_p=\sqrt[p]{\,\sum_{n=1}^\infty
s_n^p(X)}$, while $\|X\|_\infty =\|X\|=s_1(X)$ coincides with the
operator  norm $\|X\|$. Minimal and maximal u.i. norm are among
Schatten norms, i.e., $\|X\|_\infty\le|\|X\||\le\|X\|_1$ for all
$X\in\mathcal{C}_1(\mathcal{H})$ (see inequality (IV.38) in
\cite{jockosklocko}). For $f,g\in\mathcal{H}$, we will denote by
$g^*\otimes f$ one dimensional operator $(g^*\otimes f)h=\langle
h,g\rangle f$ for all $h\in\mathcal{H}$, and it is known that the linear span
of $\{g^*\otimes f\,|\, f,g\in \mathcal{H}$ is dense in each of
$\mathcal{C}_p(\mathcal{H})$ for $1\le p\le\infty$. Schatten
$p$-norms are also classical examples of $p$-reconvexized norms.
Namely, any u.i. norm  $\|.\|_\Phi$ could be
$p$-reconvexized for any $p\ge1$ by setting $\|A\|_{\Phi^{(p)}}
= \| |A|^p\|_{\Phi}^{\frac1p}$ for all $A\in \mathcal{B(H)}$ such that
$|A|^p\in \Phi(\mathcal{H})$. For the proof of the triangle
inequality and other properties of these norms see preliminary
section in \cite{joc09i} and for the characterization of the dual norm
for $p$-reconvexized one see Theorem 2.1 in \cite{joc09i}.
 The set $\mathcal{C}_{|||\cdot|||}=\{A \in \mathcal{K}(\mathcal{H}) : 
 \left\vert \left\vert \left\vert A \right\vert \right\vert \right\vert <
  \infty \}$ is a closed self-adjoint ideal $\mathcal{J}$ of $\mathcal{B}(
  \mathcal{H})$ containing finite rank operators. It enjoys the following 
  properties. First, for all $A,B\in \mathcal{B(H)}$ and $X \in \mathcal{J}$,
$
\left\vert \left\vert \left\vert AXB\right\vert \right\vert \right\vert \leq
\left\vert \left\vert A\right\vert \right\vert \ \left\vert \left\vert
\left\vert X\right\vert \right\vert \right\vert \ \left\vert \left\vert
B\right\vert \right\vert\,.
$
Secondly, if $X$ is a rank one operator, then
$
\left\vert \left\vert \left\vert X\right\vert \right\vert \right\vert =\|X\|\,.
$
The Ky Fan norm as an example of unitarily invariant norms is defined by $\| A\|
_{(k)}=\sum_{j=1}^{k}s_{j}(A)$ for $k=1,2,\ldots$. The Ky Fan
dominance theorem \cite{Con} states that $\| A\|
_{(k)}\leq \| B\| _{(k)}\,\,(k=1,2,\ldots )$ if and only if
$|||A||| \leq |||B|||$ for all unitarily invariant norms
$|||\cdot|||$, see \cite{Kry} for more information on unitarily invariant norms.
The inequalities involving unitarily invariant norms have been of special interest
(see  \cite{Con} and the references therein).

\begin{lemma}\label{Interpolacija} Let $\mathcal{T}$ and $\mathcal{S}$ be linear mappings defined
 on $\mathcal{C}_\infty(\mathcal{H}).$ If
$\|\mathcal{T}X\|\le\|\mathcal{S}X\|\mbox{ for all }X\in \mathcal{C}_\infty(\mathcal{H}),
\;\|\mathcal{T}X\|_1\le\|\mathcal{S}X\|_1\mbox{ for all }X\in
\mathcal{C}_\infty(\mathcal{H})$
then
$ \mathcal{T}X\le\mathcal{S}X$
for all unitarily invariant norms.
\end{lemma}

\begin{proof} The norms $\|\cdot\|$ and $\|\cdot\|_1$ are dual to each other in the sense that
$\|X\|=\sup_{\|Y\|_1=1}|tr(XY)|$ and $ \|X\|_1=\sup_{\|Y\|=1}|tr(XY)|.$
Hence, $\|\mathcal{T}^*X\|\le\|\mathcal{S}^*X\|$, $\|\mathcal{T}^*X\|_1\le\|\mathcal{S}^*X\|_1$.
Consider the Ky Fan norm $\|\cdot\|_{(k)}$. Its dual norm is
$\|\cdot\|_{(k)}^\sharp=\max\{\|\cdot\|,(1/k)\|\cdot\|_1\}$.
Thus, by duality, $\|\mathcal{T}X\|_{(k)}\le\|\mathcal{S}X\|_{(k)}$ and the result
 follows by Ky Fan dominance property as shown in \cite{Kry}.
\end{proof}
An operator $A\in \mathcal{B(H)}$ is called $G_{1}$ operator if the  growth condition
$$
\left\Vert (z-A)^{-1}\right\Vert =\frac{1}{{\rm{dist}}(z,\sigma (A))}
$$
holds for all $z$ not in the spectrum $\sigma (A)$ of $A$. Here ${\rm{dist}}(z,\sigma (A))$ denotes the distance between $z$ and $\sigma
(A)$. It is known that hyponormal (in particular, normal) operators are $
G_{1}$ operators \cite{jockosklocko}.
Let $A, B\in \mathcal{B(H)}$ and let $f$ be a function which is analytic on an open neighborhood $
\Omega $ of $\sigma (A)$ in the complex plane. Then $f(A)$ denotes the
operator defined on $\mathcal{H}$ by
$
f(A)=\frac{1}{2\pi i}\int\limits_{C}f(z)(z-A)^{-1}dz,  \label{4}
$
 called the Riesz-Dunford integral, where $C$ is a positively oriented simple closed rectifiable contour
surrounding $\sigma (A)$ in $\Omega $ (see \cite{joc09i} and the references therein). The spectral mapping theorem asserts that $\sigma (f(A))=f(\sigma (A))$. Throughout this note,  $\mathbb{D}=\{z\in\mathbb{C}:\left\vert z\right\vert <1\}$ denotes the unit disk, $\partial\mathbb{D}$ stands for the boundary of $\mathbb{D}$ and $d_{A}={\rm{dist}}(\partial\mathbb{D},\sigma (A))$. In addition, we adopt the notation
$\mathfrak{H}=\{f: \mathbb{D}\to \mathbb{C}: f \mbox{~is analytic}, \Re(f)>0 \mbox{~and} f(0)=1\}.$
 In this work, we present some upper bounds for $|||f(A)Xg(B)\pm X|||$, where $A, B$ are $G_{1}$ operators, $|||\cdot|||$ is a unitarily invariant norm and $f, g\in \mathfrak{H}$. Further, we find some new upper bounds for the the Schatten $2$-norm of $f(A)X\pm Xg(B)$. Up-to this juncture, we find some upper estimates for $|||f(A)Xg(B)+ X|||$ in terms of $|||\,|AXB|+|X|\,|||$ and $|||f(A)Xg(B)- X|||$ in terms of $|||\,|AX|+|XB|\,|||$, where $A, B$ are $G_{1}$ operators,  and $f, g\in \mathcal{H}$.
\begin{prop}\label{P1}
If $A,B\in \mathcal{B(H)}$ are $G_{1}$ operators with $\sigma (A)\cup \sigma (B)\subset\mathbb{D}$ and $f, g \in \mathcal{H}$, then for every $X\in \mathcal{B(H)}$ and for every unitarily invariant norm $\left\vert \left\vert \left\vert\cdot \right\vert \right\vert \right\vert $, the inequality
$
\left\vert \left\vert \left\vert f(A)Xg(B)+X\right\vert \right\vert
\right\vert \leq  \frac{2\sqrt{2}}{d_{A}d_{B}} \left\vert\left\vert \left\vert\,|AXB|+|X|\,\right\vert \right\vert \right\vert  \label{5}
$
holds.
\end{prop}

\begin{proof}
From the Herglotz representation theorem \cite{JOC} it follows that $f\in \mathcal{H}$ can be
represented as
\begin{eqnarray}
f(z)=\int\limits_{0}^{2\pi }\frac{e^{i\alpha}+z}{e^{i\alpha}-z}d\mu(\alpha)+i\Im f(0)=\int\limits_{0}^{2\pi }\frac{e^{i\alpha}+z}{e^{i\alpha}-z}d\mu(\alpha)  \label{6}
\end{eqnarray}
where $\mu $ is a positive Borel measure on the interval $[0,2\pi ]$ with
finite total mass $\int\limits_{0}^{2\pi }d\mu(\alpha)=f(0)=1$. Similarly $g(z)=\int\limits_{0}^{2\pi }\frac{e^{i\alpha}+z}{e^{i\alpha}-z}d\nu(\alpha)$ for some positive Borel measure $\nu$ on the interval $[0,2\pi ]$ with finite total mass $1$. We have
\begin{eqnarray*}
&&\hspace{-0.5in}f(A)Xg(B)+X\\
&=&\int\limits_{0}^{2\pi }\int\limits_{0}^{2\pi }\left[ \left( e^{i\alpha}-A\right)^{-1}
\left( e^{i\alpha}+A\right)X\left( e^{i\beta}+B\right) \left(
e^{i\beta}-B\right)^{-1}+X\right] d\mu(\alpha)d\nu(\beta).
\end{eqnarray*}
By some computation we have
\begin{eqnarray} \label{s1}
&&\hspace{-1.5cm}\left\vert \left\vert \left\vert f(A)Xg(B)+X\right\vert \right\vert
\right\vert \nonumber\\
&\leq& \int\limits_{0}^{2\pi }\int\limits_{0}^{2\pi }2\left\Vert \left( e^{i\alpha
}-A\right)^{-1}\right\Vert \left\vert \left\vert \left\vert AXB+
e^{i\alpha}Xe^{i\beta}\right\vert \right\vert \right\vert \left\Vert \left( e^{i\alpha
}-B\right)^{-1}\right\Vert d\mu(\alpha)d\nu(\beta).\nonumber \\
\end{eqnarray}
Since $A$ and $B$ are $G_{1}$ operators, we deduce  that
\begin{eqnarray}\label{s2}
\left\vert \left\vert \left( e^{i\alpha}-A\right)^{-1}\right\vert
\right\vert =\frac{1}{{\rm{dist}}(e^{i\alpha},\sigma (A))}\leq \frac{1}{{\rm{dist}}(\partial\mathbb{D},\sigma (A))}=\frac{1}{d_{A}},
\end{eqnarray}
and similarly
$
\left\vert \left\vert \left( e^{i\beta}-B\right)^{-1}\right\vert \right\vert \leq \frac{1}{d_{B}}.
$
Now we know that for every positive operators $C, D$, every non-negative operator monotone function $h(t)$ on $[0,\infty)$ and every unitarily invariant norm $|||\cdot|||$  it holds that $|||h(A+B)||| \leq |||h(A)+h(B)|||$.
Now from the Ky Fan dominance theorem and  we infer that
\begin{eqnarray}\label{s4}
 \left\vert\left\vert \left\vert  AXB+ e^{i\alpha}Xe^{i\beta} \right\vert \right\vert \right\vert \leq \sqrt{2} \left\vert\left\vert \left\vert\, |AXB|+|X| \, \right\vert \right\vert \right\vert.
\end{eqnarray}
Therefore, it follows from Inequality \ref{s1}, Inequality \ref{s2} and Equation \ref{s4} that $$
\left\vert \left\vert \left\vert f(A)Xg(B)+X\right\vert \right\vert
\right\vert \leq  \frac{2\sqrt{2}}{d_{A}d_{B}} \left\vert\left\vert \left\vert\,|AXB|+|X|\,\right\vert \right\vert \right\vert
$$
which completes the proof.
\end{proof}

\begin{theorem}\label{c1}
Let $f, g\in \mathcal{H}$ and $A\in\mathcal{B(H)}$ be a $G_{1}$ operator with $\sigma (A)\subset\mathbb{D}$. The inequality
$
\left\vert \left\vert \left\vert f(A)Xg(A^*)+X\right\vert \right\vert
\right\vert \leq  \frac{2}{d_{A}^2} \left\vert\left\vert \left\vert\ A|X|A^*+|X|\ \right\vert \right\vert \right\vert
$
holds for every normal operator  $X\in\mathcal{B(H)}$ commuting with $A$ and for every unitarily invariant norm $\left\vert \left\vert \left\vert\cdot \right\vert \right\vert \right\vert $.
\end{theorem}
\begin{proof}
Let $X$ and  $AXB$  be normal. Since $||| C+D |||\leq |||\,|C|+|D|\,|||$ for any normal operators $C$ and $D$, the constant $\sqrt{2}$ can be reduced to $1$ in Equation \ref{s4}.
Now from Fuglede--Putnam theorem, if $A\in \mathcal{B(H)}$ is an operator, $X\in {\mathcal(B)}({\mathcal(H)})$ is normal and $AX=XA$, then $AX^*=X^*A$. Thus if $X$ is a normal operator commuting with a $G_{1}$ operator $A$, then $AXA^*$ is normal, $|AXA^*|=A|X|A^*$ and $A^*$ is a $G_1$ operator with  $d_{A^*}=d_A$. By Proposition \ref{P1} the proof is complete.
\end{proof}
Next, letting $A=B$ in Proposition \ref{P1}, we obtain the following result.
\begin{cor}\label{C2}
Let $f, g\in \mathcal{H}$ and $A\in \mathcal{B(H)}$ be a $G_{1}$ operator with $\sigma (A)\subset\mathbb{D}$. Then
$
\left\vert \left\vert \left\vert f(A)Xg(A)-X\right\vert \right\vert
\right\vert \leq  \frac{2\sqrt{2}}{d_{A}^2} \left\vert\left\vert \left\vert \,|AX|+|XA|\,\right\vert \right\vert \right\vert
$
for every $X\in\mathcal{B}(\mathcal{H})$ and for every unitarily invariant norm $\left\vert \left\vert \left\vert
\cdot \right\vert \right\vert \right\vert $.
\end{cor}
Setting $X=I$ in in Proposition \ref{P1} again, we obtain the following result.
\begin{cor}\label{C3}
Let $f, g\in \mathfrak{H}$ and $A,B\in\mathbb{M}_n$ be $G_{1}$ matrices such that $\sigma (A)\cup \sigma (B)\subset\mathbb{D}$. Then 
$
\left\vert \left\vert \left\vert f(A)g(B)+I\right\vert \right\vert
\right\vert \leq  \frac{2\sqrt{2}}{d_{A}d_{B}} \left\vert\left\vert \left\vert\,|AB|+I\,\right\vert \right\vert \right\vert
$ for every unitarily invariant norm $\left\vert \left\vert \left\vert
\cdot \right\vert \right\vert \right\vert. $
\end{cor}
\begin{cor}\label{lem1}
If $A\in \mathcal{B}(\mathcal{H})$ is self-adjoint and $f$ is a continuous complex function on $\sigma(A)$, then $f(UAU^*)=Uf(A)U^*$ for all unitaries $U$.
\end{cor}
\begin{proof}
By the Stone-Weierstrass theorem, there is a sequence $(p_n)$ of polynomials uniformly converging to $f$ on $\sigma(A)$. Hence,
$$f(UAU^*)=\lim_np_n(UAU^*)=U(\lim_np_n(A))U^*=Uf(A)U^*.$$
We note that $\sigma(UAU^*)=\sigma(A)$.
\end{proof}
We conclude this section by presenting some inequalities involving the Hilbert-Schmidt norm $\|\cdot\|_2.$
\begin{theorem}\label{hilb}
Let $A,B\in\mathbb{M}_n$ be Hermitian matrices satisfying $\sigma(A)\cup \sigma(B)\subset \mathbb{D}$ and let $f, g\in \mathfrak{H}$. Then
$
\|f(A)X\pm Xg(B)\|_2\leq \left\|\frac{X+|A|X}{d_A}+\frac{X+X|B|}{d_B}\right\|_2.
$
\end{theorem}
\begin{proof}
Let $A=UD(\nu_j)U^*$ and $B=VD(\mu_k)V^*$ be the spectral decomposition of $A$ and $B$ and let $Y=U^*XV:=[y_{jk}].$ Noting that $|e^{i\alpha}-\lambda_j|\geq d_A$ and $|e^{i\beta}-\mu_k|\geq d_B,$ we have from \cite{L-M-R} that
\begin{align*}
\|f(A)X\pm Xg(B)\|_2^2&=\sum_{j,k}|f(\lambda_j)\pm g(\mu_k)|^2|y_{jk}|^2\\&\leq\sum_{j,k}\left(\frac{1+|\lambda_j|}{d_A}+\frac{1+|\mu_k|}{d_B}\right)^2|y_{jk}|^2\\
&=\left\|\frac{X+|A|X}{d_A}+\frac{X+X|B|}{d_B}\right\|_2^2,
\end{align*}
which completes the proof.
\end{proof}

\section{Operator valued functions}
In this section, we present some results on operator valued
 functions.
From \cite{Con}, if $(\Omega,\mathcal{M,}\mu)$ is a measure
space, for a $\sigma$-finite measure  $\mu$ on $\mathcal{M}$,
the mapping $\mathcal{A}:\Omega\rightarrow \mathcal{B(H)}$ will be called $[\mu]$ weakly$^{*}$-measurable if  a scalar
valued function $t \rightarrow tr (A_{t} Y)$ is measurable for any
$Y\in\mathcal{C}_{1}(\mathcal{H})$. Moreover, if  all these functions are in $\L^{1}(\Omega, \mu)$, then
since $\mathcal{B(H)}$ is the dual space of $\mathcal{C}_{1}(\mathcal{H})$, for
any $E\in \mathcal{M}$  we have the unique operator $I_{E}\in \mathcal{B(H)}$,
called the Gel'fand  or weak $^*$-integral of
$\mathcal{A}$ over $E$, such that
\begin{equation}
tr(\mathcal{I}_{E} Y)=\int_E tr(A_{t} Y)dt \textrm{\qquad for all
$Y\in \mathcal{C_{1}(H)}$.} \label{geljfandovintegral}
\end{equation}
We  denote it  by $\int_EA_{t}d\mu(t)$ or $\int_E A d\mu.$ We consider the following important aspect.
\begin{prop}\label{P2}

 $A:\Omega \rightarrow \mathcal{B(H)}$ is $[\mu]$
 if and only if
 scalar valued functions $t \rightarrow \langle A_{t} f,f\rangle$ are  $[\mu]$ measurable (resp. integrable) for every
$f\in \mathcal{H}$.
\end{prop}
\begin{proof}
Every one dimensional operator $f^{*} \otimes f$ is in $C_{1}(H)$ and there
holds $$tr(A_{t}( f^{*} \otimes f))=tr(f^{*} \otimes A_{t} f)=\left<A_{t} f,f\right>$$
so that $[\mu]$ weak $^*$-measurability (resp. $[\mu]$ weak
$^*$-integrability) of $A$ directly implies measurability
(resp. integrability) of $\left<A_{t} f,f\right>$ for any
$f\in \mathcal{H}$.
The converse follows immediately from \cite{jockosklocko}  and this completes the proof.
\end{proof}

We note that in view of Proposition \ref{P2},
the  Equation \ref{geljfandovintegral} of Gel'fand integral
for o.v. functions can be reformulated as follows \cite{joc09i}:

\begin{prop}\label{novadefinicionalema}

If  $\left<A f,f\right>\in L^1(E,\mu)$ for all $f\in \mathcal{H}$, for
some $E\in \mathcal{M}$ and a $\mathcal{B(H)}$-valued function $A$ on $E$, then
the mapping $f\rightarrow\int_E  \left< A_{t} f,f\right>d\mu(t)$ represents a
quadratic form of  bounded operator
$\int_E A dm$ or  $\int_E A_t d\mu(t)$,
 satisfying the following
$
  \left<\left(\int_E A_t d\mu (t)\right) f,g\right>=
 \int_E \left< At f,g\right>\,d\mu (t), for\;\; all\;\;  f,g\in \mathcal{H}.$

\end{prop}

\begin{proof}

It suffices to show that for all $E\in \mathcal{M},$
$
\Phi_E(f,g)=\int_E \left<A_{t} f,g\right>\,d\mu (t),$
for all $f, g\in \mathcal{H}$,
   defines a    bounded  sesquilinear  functional $\Phi$ on $\mathcal{H}$.
Indeed, by \cite{JOC} we have,
$
| \Phi_E(f,g)| \le \int_E|\left<A_{t} f,g\right>|\,d\mu
(t)
    \le \| A_{t} f,g\|_{L^1}
\le M \|f\|\|g\|
$
 for all $f,g\in \mathcal{H}$  since integration is a contractive functional on
$\L^{1}(\Omega ,\mu)$). This completes the proof.
\end{proof}
\begin{remark}
It is known from \cite{JOC} that
for a $[\mu]$
     $A:\Omega \rightarrow \mathcal{B(H)}$  we have that $A^*A$ is Gel'fand
integrable if and only if $ \int_\Omega \|A_t f\|^2d\mu
(t)< \infty,$ for all
$f\in \mathcal{H}$. Moreover,
for a $[\mu]$
 function
     $A:\Omega \rightarrow \mathcal{B(H)}$ let us consider a linear transformation $\vec{A}:D_{\vec{A}}\rightarrow L^{2}(\Omega,\mu , \mathcal{H})$,
with the domain $D_{\vec{A}}=\{ f\in \mathcal{H} \,| \,
  \int_\Omega \|A_t f\|^2 d\mu
(t)<\infty\}$, defined by
$
 ({\vec{A}}f)(t)=A_t f .$  and all $f\in
 D_{\vec{A}}.$
 \end{remark}
In the next section, we devote our efforts to results on inner product
type integral transformers in terms of Landau, Cauchy-Schwarz
 and Gr\"uss type norm inequalities.

\section{  Norm inequalities}
 In this section, we consider various types of norm inequalities for inner product
type integral transformers  discussed in  \cite{jockosklocko}, \cite{JOC},  \cite{joc09i} and \cite{L-M-R}.
 From \cite{JOC}, a sufficient condition is
provided when  $A^*$ and $B$ from Definition \ref{def2} are both in
$L^2_G(\Omega,d\mu, \mathcal{B(H)}).$ If each of families  $(A_t)_{t\in\Omega}$
and  $(B_t)_{t\in\Omega}$ consists of commuting normal operators,
then by Theorem 3.2 in \cite{JOC} the i.p.t.i   transformer $\int_\Omega
A_t \otimes B_t d\mu(t)$ leaves every u.i. norm ideal
$\mathcal{C}_{|\|\cdot|\|}(\mathcal{H})$ invariant and the following
Cauchy-Schwarz inequality holds:
\begin{equation}
\left|\left\| \int_\Omega A_t X B_t d\mu(t)     \right|\right\|\le \left|\left\| \sqrt{\int_\Omega
A_t^* A_t }     d\mu(t)
     \sqrt{\int_\Omega B_t^*B_t d\mu(t)}\right|\right\|,
\end{equation}
for all $X\in \mathcal{C}_{|\|\cdot|\|}(\mathcal{H})$.
Normality and commutativity condition can be dropped for Schatten
$p$-norms as shown in Theorem 3.3 in \cite{JOC}. In Theorem 3.1 in
\cite{joc09i}, a formula for the exact norm of the i.p.t.i transformer
$\int_\Omega A_t \otimes B_td\mu(t)$ acting on $\mathcal{C}_2(\mathcal{H})$ is
found. In Theorem 2.1 in \cite{joc09i}  the exact norm of the i.p.t.i
transformer $\int_\Omega A_t^* \otimes A_t d\mu(t)$ is given for two
specific cases:
\begin{equation}
\left\| \int_\Omega A_t^*\otimes A_t d\mu(t)
\right\|_{B(H)\to\mathcal{C}_{\Phi}(\mathcal{H})}= \left\| \int_\Omega A_t^*A_t d\mu(t)
\right\|_{\mathcal{C}_\Phi(\mathcal{H})}, \label{bhubiloshta}
\end{equation}
\begin{equation}
\left\| \int_\Omega A_t^*\otimes A_t d\mu(t)
\right\|_{\mathcal{C}_{\Phi}(\mathcal{H})\to\mathcal{C}_1(\mathcal{H})}= \left\| \int_\Omega
A_t A_t^* d\mu(t)    \right \|_{\mathcal{C}_{\Phi_*}(\mathcal{H})},
\nonumber
\end{equation}
where $\Phi_*$ stands for a s.g. function related to the dual space
$(\mathcal{C}_{\Phi}(\mathcal{H}))^*$.
The norm appearing in (\ref{bhubiloshta})
and its associated space $L_G^2(\Omega,d\mu,\mathcal{B(H)},\mathcal{C}_\Phi(\mathcal{H}))$
present only a  special case of norming a field
$A=(A_t)_{t\in\Omega}$. A much wider class of norms $ \|
\cdot\|_{\Phi,\Psi}$ and their associated spaces
$L_G^2(\Omega,d\mu,\mathcal{B(H)},\mathcal{C}_\Phi(\mathcal{H}))$ are
 given in
\cite{joc09i} by
\begin{equation}
  \| A\|_{\Phi,\Psi}=
\left\|\int_\Omega A_t^*\otimes A_t d\mu(t)
 \right\|_{B(\mathcal{C}_\Phi(\mathcal{H}),\mathcal{C}_\Psi(\mathcal{H}))}^\frac12
\end{equation}
for an arbitrary pair of s.g. functions $\Phi$ and $\Psi$. For the
proof of completeness of the space
 $L_G^2(\Omega,d\mu,\mathcal{C}_\Phi(\mathcal{H}, \mathcal{C}_\Psi(\mathcal{H}))$see Theorem 2.2 in
\cite{joc09i}.
Before going into the details of this section lets consider the
following proposition from \cite{L-M-R} which will be useful in the sequel.
We give its proof for completion.
\begin{prop}
\label{optlema} Let $\mu$ be a probability measure  on $\Omega$, then
for every field $(\mathscr{A}_t)_{t\in\Omega}$ in
$L^2(\Omega,\mu,\mathcal{B}(\mathcal{H}))$, for all
$B\in\mathcal{B}(\mathcal{H})$, for all unitarily invariant norms
$|\|\cdot|\|$ and for all $\theta>0$,
\begin{eqnarray}
\lefteqn{ \int_\Omega\left|\mathscr{A}_t-B\right|^2 d\mu(t)  =
 \int_\Omega\left|\mathscr{A}_t-\int_\Omega A_t d\mu(t)\right|^2 d\mu(t)
+\left| \int_\Omega A_t d\mu(t)-B\right|^2}\label{nulto}\\
 &\ge& \int_\Omega\left|\mathscr{A}_t-\int_\Omega A_t
d\mu(t)\right|^2 d\mu(t) =\int_\Omega|\mathscr{A}_t|^2
d\mu(t)-\left|\int_\Omega A_t d\mu(t)\right|^2; \label{prvo}
\end{eqnarray}
\begin{eqnarray}
 \min_{B\in\mathcal{B}(\mathcal{H})}\left|\left\|\left|\int_\Omega\left|\mathscr{A}_t-B\right|^2 d\mu(t)|\right|^\theta\right|\right\| &=&  \left|\left\| \left|\int_\Omega\left|\mathscr{A}_t-\int_\Omega A_t
d\mu(t)\right|^2 d\mu(t)\right|^\theta \right\|\right| \\
  &=& \left|\int_\Omega|\mathscr{A}_t|^2 d\mu(t)- \left|\int_\Omega A_t
d\mu(t)\right|^2\||^\theta|\right\|. \label{drugo}
\end{eqnarray}

\end{prop}
Thus, the considered minimum is always obtained for
$B=\int_\Omega A_t d\mu(t)$.

\begin{proof}
The expression in (\ref{nulto}) is trivial.
Inequality in  (\ref{prvo}) follows from (\ref{nulto}), while
identity in \eqref{prvo} is just a
 a special case of  Lemma 2.1 in \cite{JOC} applied for $k=1$ and $\delta_1=\Omega$.

As $0\le A\le B$ for $A,B\in \mathcal{C_{\infty}(H)}$ implies $ s_n^\theta(A)\le
s_n^\theta(B)$ for all $n\in \mathbb{N}$, as well as $|\| A^\theta|\|\le
|\| B^\theta|\|,$ then (\ref{drugo}) follows.
\end{proof}
Let us recall that for a pair of random real variables $(Y,Z)$ its
coefficient  of correlation
$$\rho_{Y,Z}=\frac{| E(YZ)-E(Y)E(Z)|}{\sigma(Y)\sigma(Z)}=
             \frac{| E(YZ)-E(Y)E(Z)|}{
\sqrt{E(Y^2)-E^2(Y)} \sqrt{E(Z^2)-E^2(Z)}}$$ always satisfies
$|\rho_{Y,Z}|\le 1.$ For square integrable functions $f$ and $g$ on $[0,1]$ and
$D(f,g)=\int_0^1f(t)g(t)\,d t-
        \int_0^1f(t)\,d t\int_0^1g(t)\,d t.$ Landau proved that $ | D(f,g)|\le \sqrt{D(f,f)D(g,g)}.$\\
The  following next result represents a generalization of Landau inequality
in u.i. norm ideals \cite{joc09i} for Gel'fand integrals of o.v. functions
 with relative simplicity of its formulation.

\begin{theorem}
\label{normalnisluchaj} If $\mu$ is a probability measure on
$\Omega$, let both fields $(A_t)_{t\in\Omega}$ and
 $(B_t)_{t\in\Omega}$ be in $L^2(\Omega,\mu,\mathcal{B(H)})$
 consisting of commuting normal operators
and consider
$$\sqrt{\,\int_\Omega|A_{t}|^2 -\left|\int_\Omega A_{t} d\mu(t)\right|^2}X
\sqrt{\,\int_\Omega| B_{t}|^2 d\mu(t)-\left|\int_\Omega B_{t}
d\mu(t)\right|^2},$$ for some $X\in B(H)$. Then $$\int_\Omega
A_tX B_t d\mu(t)-\int_\Omega A_{t} dt X\!\!\int_\Omega B_{t} d\mu(t)
\in C_{|\|.|\|}(H).$$
\end{theorem}

\begin{proof}
First we note that we have the following Korkine type identity for
i.p.t.i  transformers
\begin{eqnarray}
\nonumber && \int_\Omega A_t X\ B_t d\mu(t)\!-\!\int_\Omega A_{t}
d\mu(t) X\! \!     \int_\Omega B_{t} d\mu(t)
\!\!\\
\nonumber&&\lefteqn{=\!\!\int_\Omega d\mu(s)\int_\Omega  A_t X B_t d\mu(t)\!-\!\int_\Omega\!\int_\Omega A_t X B_s\,d\mu(s)d\mu(t)}\\
\!\!&&=\!\!\dfrac12\int_{\Omega^2}(A_s- A_t)X( B_s- B_t)d(\mu\times\mu)(s,t).\label{21}
\end{eqnarray}

In this representation we have $(A_s-A_t)_{(s,t)\in\Omega^2}$
and $(B_s-B_t)_{(s,t)\in\Omega^2}$ to be in
$L^2(\Omega^2,\mu\times\mu, \mathcal{B(H)})$ because by  an application of the
identity (\ref{21}),

\begin{eqnarray}
\nonumber
\dfrac12\int_{\Omega^2}\left| A_s- A_t\right|^2d(\mu\times\mu)(s,t)&=&\int_\Omega| A_{t}|^2
d\mu(t)-
\left|\int_\Omega A_{t} d\mu(t)\right|^2\\
&=&\int_\Omega\left| A_t-\int_\Omega A_{t} d\mu(t)\right|^2 d\mu(t)
\in B(H).\label{23}
\end{eqnarray}

Both families $(A_s-A_t)_{(s,t)\in\Omega^2}$ and
$(B_s-B_t)_{(s,t)\in\Omega^2}$ consist of commuting normal
operators and by Theorem 3.2 in \cite{JOC}
 $$\dfrac12\int_{\Omega^2}(A_s-A_t)X(B_s-B_t)d(\mu\times\mu)(s,t)\in \mathcal{C}_{|\|\cdot|
|\|}(\mathcal{H})$$

due to identities (\ref{22}) and (\ref{23}). And so the conclusion
(\ref{21}) follows.
\end{proof}

\begin{lemma}
Let $\mu$ (resp. $\nu$) be  a probability measure on $\Omega$ (resp.
$\mho$), let both families
$\{A_s,C_t\}_{(s,t)\in\Omega\times\mho}$ and
$\{B_s, D_t\}_{(s,t)\in\Omega\times\mho}$
 consist of commuting normal operators
and let
\begin{eqnarray*}&&\sqrt{\,\int_\Omega|A_s|^2d\mu(s)
\int_\mho|C_t|^2d\nu(t)-\left|\int_\Omega A_sd\mu(s)\int_\mho C_td\nu(t)\right|^2} X\\
&&\sqrt{\,\int_\Omega| B_s|^2d\mu(s)
\int_\mho|D_t|^2d\nu(t)-\left|\int_\Omega B_sd\mu(s)\int_\mho D_td\nu(t)\right|^2}
\end{eqnarray*}
be in $\mathcal{C}_{|\|\cdot|
|\|}(\mathcal{H})$ for some $X\in \mathcal{B(H)}$.
Then \begin{eqnarray*}&&\int_\Omega \int_\mho A_s
C_tX B_s D_t\,d\mu(s)\,d\nu(t) -\\&-&\int_\Omega A_s
\,d\mu(s)\int_\mho C_t\,d\nu(t) X\!\!\int_\Omega B_s \,d\mu(s)
\int_\mho D_t\,d\nu(t) \in\mathcal{C}_{|\|\cdot|
|\|}(\mathcal{H})\end{eqnarray*}
\end{lemma}

\begin{proof}
Apply Theorem \ref{normalnisluchaj} to the probability measure
$\mu\times\nu$ on $\Omega\times\mho$ and families
 $(A_s C_t)_{(s,t)\in\Omega\times\mho}$ and
$( B_s D_t)_{(s,t)\in\Omega\times\mho}$ of normal commuting
operators in $L_G^2(\Omega\times\mho,d\mu\times\nu,\mathcal{B(H)}).$ The rest follows trivially.
\end{proof}

Next we consider Landau type inequality for i.p.t.i   transformers in Schatten
ideals for the Schatten $p$-norms.

\begin{prop}
Let $\mu$ be a probability measure on $\Omega$, let
$(A_t)_{t\in\Omega}$ and $(B_t)_{t\in \Omega}$ be
$\mu$-weak${}^*$ measurable families of bounded Hilbert space
operators such that\\
$\int_\Omega\left(\|A_t f\|^2+\|A_t^* f\|^2+\| B_t f\|^2+\|B_t^* f\|^2\right)d\mu(t)<\infty\;
\textrm{ \rm for all $f\in \mathcal{H}$}$ and let $p,q,r\ge1$ such that
$\dfrac1p=\dfrac1{2q}+\dfrac1{2r}\,$. Then for all
$X\in \mathcal{C}_p(\mathcal{H})$,
\begin{eqnarray}
\label{grussp}&&\lefteqn{
           \left\|\int_\Omega A_t X B_t d\mu(t)-      \int_\Omega A_{t} d\mu(t)   X \int_\Omega B_{t} d\mu(t)\right\|_p}\\
         &\leqslant&\kern-4pt\left\|
\left(\int_\Omega\left|\left(\int_\Omega\left|A_t^*-\int_\Omega A_{t}^*
d\mu(t) \right|^2 d\mu(t)   \right)^{\frac{q-1}2}
\left(A_t-\int_\Omega A_{t} d\mu(t)\right)\right|^2
d\mu(t)\right)^{\frac1{2q}}\right\|\nonumber\end{eqnarray}
\begin{equation*} X\left\|\left(\int_\Omega\left|\left(\int_\Omega\left| B_t-\int_\Omega B_{t}d\mu(t)
\right|^2 d\mu(t)   \right)^{\frac{r-1}2}
\left( B_t^*-\int_\Omega B_{t}^* d\mu(t)\right)\right|^2
d\mu(t)\right)^{\frac1{2r}}\right\|_p.
\end{equation*}
\end{prop}

\begin{proof}
According to identity (\ref{23}), application of Theorem 3.3 in
\cite{JOC} to   families
$(\mathscr{A}_s-\mathscr{A}_t)_{(s,t)\in\Omega^2}$ and
 $(\mathscr{B}_s-\mathscr{B}_t)_{(s,t)\in\Omega^2}$ gives

\begin{eqnarray}
&&
 \left\|\int_\Omega A_t X B_t d\mu(t)-\int_\Omega\mathscr{A}_td\mu(t) X\int_\Omega\mathscr{B}_td\mu(t)\right\|_p \nonumber\\
&=&\left\|\dfrac12\int_{\Omega^2}(A_s-A_t)X(B_s-B_t)d(\mu\times\mu)(s,t)\right\|_p\le \nonumber
\end{eqnarray}

$$
\left\|\left(\dfrac12\int_{\Omega^2}(\mathscr{A}_s^*-\mathscr{A}_t^*)
\left(\dfrac12\int_{\Omega^2}|\mathscr{A}_s^*-\mathscr{A}_t^*|^2(\mu\times\mu)(s,t)\right)^{q-1}\kern-13.1pt
(\mathscr{A}_s-\mathscr{A}_t)d(\mu\times\mu)(s,t)\right)^{\frac1{2q}}\right\|.\kern-10pt
$$ \begin{equation}\label{pnorm}\end{equation}
$$\left\|\left(\dfrac12\int_{\Omega^2}(\mathscr{B}_s-\mathscr{B}_t)
\Bigl(\dfrac12\int_{\Omega^2}|\mathscr{B}_s-\mathscr{B}_t|^2(\mu\times\mu)(s,t)\Bigr)^{r-1}\kern-8pt
(\mathscr{B}_s^*-\mathscr{B}_t^*)d(\mu\times\mu)(s,t)\right)^{\frac1{2r}}\kern-2pt\right\|_p.$$

By application of identity (\ref{23}) once again, the last
expression in (\ref{pnorm}) becomes

$$\bigg\|\biggl(\dfrac12\int_{\Omega^2}(\mathscr{A}_s-\mathscr{A}_t)^*
\left(\int_\Omega\left|\mathscr{A}_t^*-\int_\Omega \mathscr{A}^*
_t d\mu(t)\right|^2d\mu(t)\right)^{q-1}
(\mathscr{A}_s-\mathscr{A}_t)d(\mu\times\mu)(s,t)\biggr)^{\frac1{2q}}$$

$$\biggl(\dfrac12\int_{\Omega^2}(\mathscr{B}_s-\mathscr{B}_t)
\Bigl(\int_\Omega\left|\mathscr{B}_s-\int_\Omega \mathscr{B}
_t d\mu(t)\right|^2d\mu(s)\Bigr)^{r-1}\kern-5pt(\mathscr{B}_s-\mathscr{B}_t)^*d(\mu\times\mu)(s,t)\biggr)^{\frac1{2r}}\bigg\|_p.\label{odvizraz}
$$
Denoting
$\Bigl(\int_\Omega\left|A_s^*-\int_\Omega\mathscr{A}^*d\mu\right|^2d\mu(s)\Bigr)^{\frac{p-1}2}$
\kern-6.5pt(resp.
$\Bigl(\int_\Omega\left|B_s-\int_\Omega\mathscr{B}d\mu\right|^2d\mu(s)\Bigr)^{\frac{r-1}2}$)
by $Y$ (resp. $Z$),
then 
the expression in (\ref{pnorm}) becomes
\begin{eqnarray}\label{saYZ}
\biggl\|\left(\dfrac12\int_{\Omega^2}\left|Y A_s-Y A_t\right|^2d(\mu\times\mu)(s,t)\right)^{\frac1{2q}}.\\
\nonumber
\left(\dfrac12\int_{\Omega^2}\left|Z B_s^*-Z B_t^*\right|^2d(\mu\times\mu)(s,t)\right)^{\frac1{2r}}\biggr\|_p.
\end{eqnarray}

By a new application of identity (\ref{23}) to families
$(Y A_t)_{t\in\Omega}$ and $(Z B_t^*)_{t\in\Omega}$ (\ref{saYZ})
becomes

$$\left\|\left(\int_\Omega\left|Y\mathscr{A}_t-\int_\Omega Y\mathscr{A}_t d\mu(t)\right|^2d\mu(t)\right)^{\frac1{2q}}\kern-4pt .
\left(\int_\Omega\left|Z\mathscr{B}_t^*-\int_\Omega Z\mathscr{B}_t^*
d\mu(t)\right|^2d\mu(t)\right)^{\frac1{2r}}\right\|_p,$$
which obviously equals to the righthand side expression in
(\ref{grussp}).
\end{proof}
The following next result from \cite{JOC} is a special case of an abstract H\"older inequality presented in
Theorem 3.1.(e) in \cite{JOC} for Cauchy-Schwarz  inequality for o.v. functions in u.i. norm ideals. We state it as follows.

\begin{prop}
\label{koshishvarcovacha} Let $\mu$ be a  measure on $\Omega$, let
$(A_t)_{t\in\Omega}$
 and $(B_t)_{t\in \Omega}$ be
$\mu$-weak${}^*$ measurable  in $\mathcal{B(H)}$
 such that
$|\int_\Omega|A_t|^2 d\mu(t)|^\theta$ and
$|\int_\Omega|B_t|^2 d\mu(t)|^\theta$ are in  $\mathcal{C}_{\||.|\|}\mathcal{H}$ for some
$\theta>0$ and for  u.i. norm. Then the following  holds.
$$
           \|||\int_\Omega A_t^* B_t d\mu(t)\|||^\theta \|||\le
           \|||\int_\Omega A_t^* A_t d\mu(t)\|||^\theta \|||^\frac12
          \|||\int_\Omega B_t^* B_t d\mu(t)\|||^\theta \|||^\frac12.
$$
\end{prop}
\begin{proof}
Take   $\Phi$ to be a
     s.g. function that generates u.i. norm $\||\cdot\||$,
 $\Phi_1=\Phi$,
$\Phi_2=\Phi_3=\Phi^{(2)}$ (2-reconvexization  of $\Phi$),
$\alpha=2\theta$ and $X=I$, and then apply 3.1 from \cite{JOC}.
\end{proof}
At this point, we give another generalization of Landau
inequality  for Gel'fand integrals of o.v. functions in u.i. norm ideals
\begin{theorem}
\label{korelacionateorema} If  $\mu$ is  a probability  measure on
$\Omega$, $\theta>0$ and $(A_t)_{t\in\Omega}$
 and $(B_t)_{t\in \Omega}$ are as
   in Proposition \ref{koshishvarcovacha},
$\mu$-weak${}^*$ measurable families of bounded Hilbert space operators
 such that
$\|||\int_\Omega|A_t|^2d\mu(t)\|||^\theta$ and
$\|||\int_\Omega|B_t|^2d\mu(t)\|||^\theta$ are 
in $\mathcal{C}_{\||.|\|}\mathcal{H}$ for some $\theta>0$ and for some u.i. norm $\||\cdot\||$
then,

  $ \left\|\left|\int_\Omega A_t^* B_td\mu(t)
                  -\int_\Omega A_t^*d\mu(t)\int_\Omega B_td\mu(t) \|||^\theta  \right\|\right|^2\le \||\int_\Omega  \||| A_t \|||^2d\mu(t)-\label{korelaciona}$
                  \\
           $  \|||\int_\Omega A_td\mu(t) \|||^2 \|||^\theta \|||
            \|||\int_\Omega  \||| B_t \|||^2d\mu(t)- \|||\int_\Omega B_td\mu(t) \|||^2 \|||^\theta \||.$

\end{theorem}
\begin{proof}
It suffices to invoke Proposition \ref{koshishvarcovacha}  to o.v. families
$(A_s-A_t)_{(s,t)\in\Omega^2}$ and
$(B_s-B_t)_{(s,t)\in\Omega^2}$ and use identity in \cite{L-M-R} and the proof is complete.
\end{proof}
Now we consider some interesting quantities that relate to norm inequalities.
For bounded set of operators $A=(\mathscr{A}_t)_{t\in\Omega}$
we see that  the radius of the smallest disk that
essentially contains its range is
$$r_\infty(A)=\inf_{A\in \mathcal{B(H)}}ess \sup_{t\in\Omega}\| A_t-A\|=
\inf_{A\in \mathcal{B(H)}}\| A_t-A\|_\infty=\min_{A\in \mathcal{B(H)}}\| A_t-A\|_\infty.$$
From the triangle inequality we have
$\bigl|\|\mathscr{A}_t-A'\|-\|\mathscr{A}_t-A\|\bigr|\leq\|A'-A\|$,
so the mapping $A\to ess \sup_{t\in\Omega}\|A_t-A\|$ is nonnegative
and continuous on $\mathcal{B(H)}$. Since $(\mathscr{A}_t)_{t\in\Omega}$ is
bounded field of operators, we also have $\| A_{t}-A\|\to\infty$
when $\|A\|\to\infty$, so this mapping attains minimum \cite{Con}, and it
actually  attains at some $A_0\in \mathcal{B(H)}$, which represents a center of
the disk considered \cite{Kry}. Any such field of operators is of finite
diameter, therefore, we have that
$r_\infty(A)=ess \sup_{s,t\in\Omega}\| A_s-A_t\|,$ with the simple
inequalities given as $r_\infty(A)\le diam_\infty(A)\le 2r_\infty(A)$
relating those quantities. For such
 fields of operators we can now state the following stronger
version of Gr\"uss inequality whose proof is found in \cite{joc09i}.
\begin{lemma}
\label{th0} Let  $\mu$ be a $\sigma$-finite measure on $\Omega$ and
 let $A=(\mathscr{A}_t)_{t\in\Omega}$ and $B=(\mathscr{B}_t)_{t\in\Omega}$
be $[\mu]$ a.e. bounded  fields of operators. Then for all
$X\in \mathcal{C_{|\|.|\|}(H)}$,
$
\sup_{\mu(\delta)>0}\||\frac1{\mu(\delta)}\int_\delta\mathscr{A}_tX\mathscr{B}_t d\mu(t) -
\frac1{\mu(\delta)}\int_\delta\mathscr{A}_t d\mu(t) \,X
\frac1{\mu(\delta)}\int_\delta\mathscr{B}_t d\mu(t) |\|\le \min_{i} \mathcal{P_{i}}\cdot\|| X\||.
\label{oshtrina0}$
(i.e. $\sup$ is taken over all measurable sets
$\delta\subseteq\Omega$ such that $0<\mu(\delta)<\infty$).
\end{lemma}
Lemma \ref{th0} has an immediate implication as seen in the next theorem when
$(\mathscr{A}_t)_{t\in\Omega}$ and $(\mathscr{B}_t)_{t\in\Omega}$
are bounded fields of self-adjoint
operators.
\begin{theorem}
\label{th1} If $\mu$ is a probability measure on $\Omega$, let
$C,D,E,F$ be bounded self-adjoint operators and let
$(\mathscr{A}_t)_{t\in\Omega}$ and $(\mathscr{B}_t)_{t\in\Omega}$ be
bounded self-adjoint fields satisfying $C\le\mathscr{A}_t\le D$ and
$E\le\mathscr{B}_t\le F$ for all $t\in\Omega$. Then for all
$X\in \mathcal{C_{|\|.|\|}(H)}$,
\begin{equation}
\left\|\left|\int_\Omega\mathscr{A}_tX\mathscr{B}_t d\mu(t)-
\int_\Omega\mathscr{A}_td\mu(t) \,X \int_\Omega\mathscr{B}_t d\mu(t)\right|\right\|
\le\dfrac{\|D-C\|\cdot\|F-E\|}4\cdot \|| X|\|. \label{oshtrina}
\end{equation}
\end{theorem}

\begin{proof}
t
As $\frac{C-D}2\le\mathscr{A}_t-\frac{C+D}2\le\frac{D-C}2$ for every
$t\in\Omega$, then
\begin{eqnarray*}
ess \sup_{t\in\Omega}\| \mathscr{A}_t-\frac{C+D}2\|&=&
ess \sup_{t\in\Omega}\sup_{\|
f\|=1}\||\langle\mathscr{A}_t-\frac{C+D}2 \| f,f\rangle|\\
&\le& \sup_{\| f\|=1}|\langle\frac{D-C}2 f,f\rangle|= \frac{\|
D-C\|}2,
\end{eqnarray*}  
which implies
 $r_\infty(A)\le
\frac{\| D-C\|}2,$ and          similarly

$r_\infty(B)\le\frac{\| F-E\|}2.$ Thus, \eqref{oshtrina} follows
directly from
    (\ref{oshtrina0}).
\end{proof}
 In case of $\mathcal{H}=\mathbb{C}$  and $\mu$ being the  normalized
Lebesgue measure on $[a,b]$ (i.e. $d\,\mu(t)=\frac{dt}{b-a}$), then
(\ref{grisovaca}) comes as an obvious corollary of  Theorem
\ref{th1}. This special case also confirms the sharpness of the
constant $\frac14$ in the inequality (\ref{oshtrina}).

Lastly, we consider, the  Gr\"uss type inequality for elementary operators in the example below.
 \begin{example}
 Let $A_1,
 \hdots, A_n$, $B_1, \hdots, B_n$, $C, D, E$ and $F$ be bounded linear self-adjoint operators
 acting on a Hilbert space $\mathcal{H}$
 such that
 $C\le A_i\le D$ and
$E\le B_i\le F$ for all $i=1,2,\cdots,n$ then for arbitrary
$X\in\mathcal{C}_{\||.|\|}\mathcal{H}$,
\begin{eqnarray}\nonumber
\||  \frac1n\sum_{i=1}^n A_i XB_i-\frac1{n^2}\sum_{i=1}^nA_i\,
X\sum_{i=1}^nB_i\|| \leq \frac{\|D-C\|\|F-E\|}4  \|| X\||.
\end{eqnarray}
\end{example}

Indeed, it is sufficient to prove that the elementary operator is normally represented and that Gr\"uss type inequality
holds for it in which case is provided in \cite{Oke1}.

In the next section we dedicate our effort to the applications of this study to other fields.
 We consider quantum theory in particular whereby we describe the application in  quantum
 chemistry and  quantum mechanics.
\section{Applications in quantum theory}
Norm inequalities and other properties of i.p.t.i transformers have various applications
in other fields. We discuss the applications in  quantum theory involving two cases \cite{Oke1}.
The first case is in quantum chemistry whereby we consider the Hamiltonian which is
 a bounded, self-adjoint operator on
some infinite-dimensional Hilbert space which governs a quantum chemical system.
 The Hamiltonian helps in estimation of ground state energies of chemical systems via subsystems.

The second case in quantum mechanics deals with commutator approximation.
The discussions of approximation by commutators $AX-XA$ or by generalized
commutator $AX-XB$ originates from  quantum
theory. For instance, the Heisenberg uncertainly principle may be mathematically deduced as
 saying that there exists a pair $A,X$ of linear operators and a non-zero
scalar $\alpha$ for which
$AX - XA = \alpha I$.  A natural question immediately arises:
How close can $AX - XA$ be to the identity?
In \cite{Oke1}, it is discussed that if $A$ is normal, then, for all $X \in B(H)$,
$||I - (AX - XA)|| \geq ||I||.$ In the   inequality here,  the zero commutator is a
 commutator approximant in $B(H)$.
\section*{Acknowledgement}The  author's appreciations go to TWAS-DFG for the financial support Grant No. 1603991000.


\begin{thebibliography}{12}
\addcontentsline{toc}{section}{References}

\bibitem{Con}\label{Con}
J.B. Conway, A Course in Functional Analysis, second ed., Springer-Verlag, New York, 1990.

\bibitem{Kry}\label{Kry}
E. Kreyzig, Introductory Functional Analysis with Applications, John Wiley
and sons, New York, 1978.
\bibitem{jockosklocko}
D.R. Joci\'c, The Cauchy-Schwarz norm  inequality for elementary operators
in Schatten  ideals, J. London. Math. Soc. 60 (1999), 925--934.

\bibitem{JOC}
D.R. Joci\'c, Cauchy--Schwarz norm inequalities for weak*-integrals
of operator valued functions, J. Funct. Anal. 218 (2005), 318--346.

\bibitem{joc09i}
D.R. Joci\'c, Interpolation norms between row and column
spaces and the norm problem for elementary operators, Linear Alg.
Appl. 430 (2009), 2961--2974.

\bibitem{L-M-R}
X. Li, R.N. Mohapatra, R.S. Rodriguez, Gr\"uss-type
inequalities, J. Math. Anal. Appl. 267 (2002),
434--443.

\bibitem{Oke1}\label{Oke1}
 N.B. Okelo, J.O.  Agure,   P.O. Oleche,   Various Notions Of
 Orthogonality in Normed Spaces, Acta Mathematica Scientia, 33 (2013), 1387--1397.

\end{thebibliography}
\end{document}